\documentclass[a4paper, 10pt]{amsart}
\usepackage[utf8]{inputenc}
\usepackage[T1]{fontenc}
\usepackage{lmodern}
\usepackage[english]{babel}
\usepackage[dvipsnames]{xcolor}
\usepackage{adjustbox}
\usepackage{letltxmacro}
\usepackage{todonotes}
\LetLtxMacro\todonotestodo\todo
\renewcommand{\todo}[2][]{\todonotestodo[#1]{TODO: {#2}}}
\usepackage{enumerate}
\usepackage{hyperref}
\usepackage{url}
\usepackage{tikz-cd}
\usepackage{verbatim}

\newtheorem{theorem}{Theorem}

\makeatletter
\newtheorem*{rep@theorem}{\rep@title}
\newcommand{\newreptheorem}[2]{%
\newenvironment{rep#1}[1]{%
 \def\rep@title{#2 \ref{##1}}%
 \begin{rep@theorem}}%
 {\end{rep@theorem}}}
\makeatother

\newreptheorem{theorem}{Theorem}
\newtheorem{theorem:intro}{Theorem}

\newtheorem{lemma}{Lemma}[section]
\newtheorem{proposition}[lemma]{Proposition}

\theoremstyle{definition}
\newtheorem{remark}[lemma]{Remark}

\newtheorem*{claim*}{Claim}

\newtheorem*{theorem*}{Theorem}
\newtheorem*{corollary*}{Corollary}
\newtheorem*{lemma*}{Lemma}

\title{Equivalent characterizations \\of non-Archimedean uniform spaces}

\author{Daniel Windisch}
\address{Institut für Analysis und Zahlentheorie\\Technische Universität Graz\\
  Kopernikusgasse 24/II\\8010 Graz\\Austria}
\email{\href{mailto:dwindisch@math.tugraz.at}{dwindisch@math.tugraz.at}}
\thanks{D.~Windisch is supported by the Austrian Science Fund (FWF): P~30934}
\thanks{\textit{Mathematics subject classification:} primary: 54E15; secondary 13J99, 22A99}
\thanks{\textit{Key words and phrases:} uniform space, non-Archimedean uniformity, non-Archimedean pseudo-metric, separation axiom, metrizability, $I$-adic topology}

\begin{document}
\maketitle

\begin{abstract}
In this paper, we deal with uniform spaces whose diagonal uniformity admits a basis consisting of equivalence relations. Such \textit{non-Archimedean uniform spaces} are particularly interesting for applications in commutative ring theory, because uniformities stemming from valuations or directed systems of ideals are of this type.\\
In general, apart from diagonal uniformities, there are two further approaches to the concept of uniform spaces: covering uniformities and systems of pseudo-metrics. For each of these ways of defining a uniformity, we isolate a non-Archimedean special case and show that these special cases themselves are equivalent to the notion of non-Archimedean diagonal uniformities.\\ 
Moreover, we formulate a seperation axiom for topological spaces $X$ that tells exactly when the topology on $X$ is induced by a non-Archimedean uniformity. In analogy to the classical metrizability theorems, we characterize when a non-Archimedean uniformity comes from a single pseudo-metric.
\end{abstract}

\section{Introduction}

Throughout this work, let $X$ be a set. Recall that a \textit{diagonal uniformity} $\mathfrak{D}$ on $X$ is a non-empty collection of relations (called \textit{entourages}) on $X$ satisfying certain axioms. A basis $\mathcal{B}$ of $\mathfrak{D}$ is a subset of $\mathfrak{D}$ such that every entourage in $\mathfrak{D}$ contains an element of $\mathcal{B}$. For a general introduction to uniform spaces, see~\cite{Willard}.\\
Following Monna~\cite{Monna}, who noticed that a non-Archimedean metric induces the following special type of uniformity, a diagonal uniformity $\mathfrak{D}$ on $X$ is called \textit{non-Archimedean}, if it has a basis $\mathcal{B}$ consisting of equivalence relations. In many settings, especially in commutative algebra and Abelian group theory, non-Archimedean uniformities are used extensively. For instance, the completion of a commutative ring $R$ with respect to a valuation or a directed system of ideals is a special case of the more general concept of completions of uniform spaces, and the original uniformity defined on $R$ is non-Archimedean (cf.~\cite[Chapter 10]{Atiyah}). In this context, one often freely uses the correspondence of diagonal uniformities, covering uniformities and systems of pseudometrics. So the following question arises naturally:\\

\noindent
\textbf{Question.} Is it possible to characterize those covering uniformities and systems of pseudo-metrics, respectively, that correspond to non-Archimedean diagonal uniformities?\\

One goal of this article is therefore to isolate special cases in the definitions of uniform structures via uniform covers and systems of pseudo-metrics, respectively, such that these special cases are equivalent to the notion of non-Archimedean diagonal uniformities. A main result is

\newtheorem{theorem1}{Theorem}
\begin{theorem1}
For a diagonal uniformity $\mathfrak{D}$ on $X$ the following are equivalent:
\begin{itemize}
\item[(A)] $\mathfrak{D}$ is non-Archimedean.
\item[(B)] There exists a system $M$ of non-Archimedean pseudo-metrics on $X$ that induces the uniformity $\mathfrak{D}$ (see Remark \ref{Remark:metrics}).
\item[(C)] The corresponding covering uniformity $\mu_\mathfrak{D}$ of $\mathfrak{D}$ (see Remark \ref{Remark:coverings}) has a basis consisting of partitions of $X$.
\end{itemize}
\end{theorem1}

It is well-known that uniformizability for a topological space is equivalent to T$_{3 \frac{1}{2}}$ (Willard~\cite{Willard} uses the term \textit{completely regular} for T$_{3 \frac{1}{2}}$). We give a similar theorem for non-Archimedean uniformizability using a stronger version of this seperation axiom:

\newtheorem{theorem2}[theorem1]{Theorem}
\begin{theorem2}
A topological space $X$ is uniformizable by a non-Archimedean uniformity if and only if it satisfies $\text{T}_\text{A}$, i.e. for every closed subset $A \subseteq X$ and every $x \in X \setminus A$ there are open sets $U_1,U_2 \subseteq X$ such that 
\begin{itemize}
\item[(i)] $U_1 \cap U_2 = \emptyset $, $U_1 \cup U_2 = X$,
\item[(ii)] $A \subseteq U_1$, $x \in U_2$.
\end{itemize}
\end{theorem2}

\noindent
Moreover, we are able to characterize the case when a non-Archimedean uniform space is pseudo-metrizable by a single non-Archimedian pseudo-metric. Results closely related to ours were given by Monna~\cite[Th\'{e}or\`{e}me 13]{Monna}, Banaschewski~\cite[Satz 6]{Banaschewski} and De Groot \cite[Theorem II]{Groot}, but these authors assume uniformities to be seperating (i.e., the induced topology is Hausdorff) and therefore formulate stronger results dealing with metrizability. Our theorem is similar to the general case, where a uniform space is pseudo-metrizable if and only if there is a countable system of pseudo-metrics inducing the uniformity (cf.~\cite[Theorem 38.3]{Willard}), but the method used there does not apply to our set-up.

\newtheorem{theorem3}[theorem1]{Theorem}
\begin{theorem3}
A uniformity is pseudo-metrizable by a single non-Archimedean pseudo-metric if and only if it is induced by a countable system of non-Archimedean pseudo-metrics.
\end{theorem3}

\vspace{0.5cm}

\begin{center}
\begin{tabular}{ |p{3.9cm}||p{4.5cm}|p{4.5cm}|  }
 \hline
 \multicolumn{3}{|c|}{Overview of Results} \\
 \hline
  & General uniform space & non-Arch. uniform space \\
 \hline
  Diagonal uniformity  & Basis of relations  &Basis of equivalence rel. \\
   & & \\
 Covering uniformity&   Basis of covers& Basis of partitions   \\
   & & \\
 Pseudo-metrics &System of pseudo-metrics & System of non-Arch. pseudo-metrics\\
 \hline
 Uniformizability   &$\text{T}_{3\frac{1}{2}}$  & $\text{T}_\text{A}$ \\
   & & \\
 Pseudo-metrizability&   countable system of pseudo-metrics  & countable system of non-Arch. pseudo-metrics\\
 
 \hline
\end{tabular}
\end{center}

\vspace{0.5cm}

\section{Equivalent approaches to non-Archimedean uniform spaces}

\subsection{Covering uniformities}

The isolation of the desired special case is very easy in the setting of covering uniformities, as we can explicitly write down a bijection between diagonal uniformities and covering uniformities on $X$.

\begin{remark} \label{Remark:coverings}
How to construct a covering uniformity out of a diagonal uniformity and vice versa.

\begin{itemize}
\item[(1)] Let $\mathfrak{D}$ be a diagonal uniformity on $X$, and let $\mathcal{B}$ be a basis of $\mathfrak{D}$. For $x \in X$ and $D \in \mathcal{B}$, let $D[x] := \{ y \in X \mid (x,y) \in D \}$ and $\mathcal{U}_D := \{ D[x] \mid x \in X \}$. Then $\beta_\mathfrak{D} := \{ \mathcal{U}_D \mid D \in \mathcal{B} \}$ is a basis of a covering uniformity $\mu_\mathfrak{D}$ on $X$, and $\mu_\mathfrak{D}$ is independent of the choice of $\mathcal{B}$.
\item[(2)] Let $\beta$ be a basis of a covering uniformity $\mu$ on $X$. For $\mathcal{U} \in \beta$, we define $D_\mathcal{U} := \{ (x,y) \in X \times X \mid \exists U \in \mathcal{U}  :  x,y \in U \}$. Then $\mathcal{B}_\mu := \{ D_\mathcal{U} \mid \mathcal{U} \in \beta \}$ is a basis of a diagonal uniformity $\mathfrak{D}_\mu$ on $X$, and $\mathfrak{D}_\mu$ is independent of the choice of $\beta$.
\item[(3)] The maps $\mathfrak{D} \mapsto \mu_\mathfrak{D}$, $\mu \mapsto \mathfrak{D}_\mu$ are bijections of the set of all diagonal uniformities on $X$ and the set of all covering uniformities on $X$ which are inverse to each other.
\end{itemize}
\end{remark}

\begin{lemma}
\begin{itemize}
\item[(1)] If $\mathfrak{D}$ is non-Archimedean, then $\mu_\mathfrak{D}$ has a basis consisting of partitions.
\item[(2)] If $\mu$ has a basis consisting of partitions, then $D_\mu$ is non-Archimedean.
\item[(3)] The maps in remark 2.1.1 restrict to bijections of non-Archimedean diagonal uniformities and covering uniformities having a basis consisting of partitions.
\end{itemize}
\end{lemma}

\begin{proof}
\begin{itemize}
\item[(1)] Let $\mathcal{B} \subseteq \mathfrak{D}$ be a basis consisting of equivalence relations. Then, for all $D \in \mathcal{B}$, we have that $\mathcal{U}_D = \{D[x] \mid x \in X \}$ is a partition of $X$.
\item[(2)] Let $\beta \subseteq \mu$ be a basis consisting of partitions of $X$. For $\mathcal{U} \in \beta$, $D_\mathcal{U} = \{ (x,y) \in X \times X \mid \exists U \in \mathcal{U}  :  x,y \in U \}$ is an equivalence relation.
\item[(3)] is an immediate consequence of (1) and (2).
\end{itemize}
\end{proof}

\subsection{Pseudo-metrics}

\noindent
The systems of pseudo-metrics on $X$ are not in general in one-to-one correspondence with the diagonal uniformities on $X$. Anyway, we can consider systems of pseudo-metrics up to equivalence:

\begin{remark} \label{Remark:metrics}
\begin{itemize}
\item[(1)] Let $M$ be a system of pseudo-metrics on $X$. For every $d \in M$ and every $\varepsilon \in \mathbb{R}_{>0}$ we consider the binary relation $D_\varepsilon^d := \{ (x,y) \in X \times X \mid d(x,y) < \varepsilon \}$. The set $\mathcal{B}_M := \{ D_\varepsilon^d \mid d \in M, \ \varepsilon \in \mathbb{R}_{>0} \}$ is a basis of a diagonal uniformity on $X$, denoted by $\mathfrak{D}_M$. 
\item[(2)] Two systems $M$ and $N$ of pseudo-metrics on $X$ are said to be \textit{equivalent}, if $\mathfrak{D}_M = \mathfrak{D}_N$. 
\item[(3)] A pseudo-metric $d$ on $X$ is said to be \textit{non-Archimedean} if it satisfies the non-Archimedean triangle inequality, i.e. for all $x,y,z \in X$ we have
\begin{align*}
d(x,y) \leq max(d(x,z),d(z,y)).
\end{align*}
\end{itemize}
\end{remark}

\begin{proposition}
\begin{itemize}
\item[(1)] If $d_1$ and $d_2$ are non-Archimedean pseudo-metrics on $X$, then so is $d = sup(d_1,d_2)$.
\item[(2)] If $M$ is a system of non-Archimedean pseudo-metrics, then $\mathcal{B}_M$ consists of equivalence relations. In particular, $\mathfrak{D}_M$ is non-Archimedean.
\end{itemize}
\end{proposition}

\begin{proof}
\begin{itemize}
\item[(1)] Let $x,y,z \in X$. Then we have
\begin{align*}
d(x,y) &= max(d_1(x,y),d_2(x,y)) \\
		&\leq max(max(d_1(x,z),d_1(z,y)),max(d_2(x,z),d_2(z,y))) \\
		&= max(d_1(x,z),d_2(x,z),d_1(z,y),d_2(z,y)) \\
		&= max(max(d_1(x,z),d_2(x,z)),max(d_1(z,y),d_2(z,y))) \\
		&= max(d(x,z),d(z,y)).
\end{align*}
So $d$ is non-Archimedean.

\item[(2)] Let $\varepsilon \in \mathbb{R}_{>0}$ and $d \in M$. Clearly, $D_\varepsilon^d$ is reflexive and symmetric. So let $(x,z),(z,y) \in D_\varepsilon^d$. Then $d(x,y) \leq max(d(x,z),d(z,y)) < \varepsilon$, which means that $(x,y) \in D_\varepsilon^d$.
\end{itemize}
\end{proof}

\noindent
The next proposition in particular shows that every system of pseudo-metrics, that induces a non-Archimedean diagonal uniformity, is equivalent to a system of non-Archimedean pseudo-metrics. In the proof, we give an explicit construction of a system of pseudo-metrics using a basis consisting of equivalence relations which is very intuitive. In the general case, the construction of pseudo-metrics based on relations is much more involved, since the relations need not even be symmetric.

\begin{proposition}
If $\mathfrak{D}$ is a non-Archimedean diagonal uniformity on $X$, then there exists a system $M$ of non-Archimedean pseudo-metrics on $X$ such that $\mathfrak{D} = \mathfrak{D}_M$.
\end{proposition}

\begin{proof}
Let $\mathcal{B} \subseteq \mathfrak{D}$ be a basis consisting of equivalence relations and consider the set
\begin{align*}
AK(\mathcal{B}) := \{ (D_n)_{n \in \mathbb{N}} \mid D_1 = X \times X \land \forall n \in \mathbb{N}_{\geq 2} : D_n \in \mathcal{B} \land \forall n \in \mathbb{N} : D_{n+1} \subseteq D_n \}
\end{align*}
of all countable descending chains in $\mathcal{B}$ starting at $X \times X$. \\
For $\kappa = (D_n)_{n \in \mathbb{N}} \in AK(\mathcal{B})$, we define
\begin{align*}
d_\kappa : X \times X &\to \mathbb{R}_{\geq 0} \\
			(x,y)     &\mapsto \begin{cases}
								0 \text{ if } \forall n \in \mathbb{N}: (x,y) 																			\in D_n\\
								\frac{1}{n} \text{ if } n = max\{ m \in 														\mathbb{N} \mid (x,y) \in D_m \}.
								\end{cases} 
\end{align*}
We first show that $d_\kappa$ is a non-Archimedean pseudo-metric for all $\kappa = (D_n)_{n \in \mathbb{N}}~\in~AK(\mathcal{B})$. Let $x,y,z \in X$. Clearly, $d_\kappa(x,x) = 0$. Since all relation in $\mathcal{B}$ are symmetric, we have $d_\kappa(x,y) = d_\kappa(y,x)$. \\
For the proof of the non-Archimedean triangle inequality, first assume that $d_\kappa(x,y) = d_\kappa(y,z) = 0$. Since all $D_n$ are transitive, we have that $d_\kappa(x,z) = 0$. If $d_\kappa(x,y)= 0$ and $d_\kappa(y,z) = \frac{1}{l}$ for some $l \in \mathbb{N}$, then by the transitivity of $D_l$ we have $(x,z) \in D_l$. Therefore $d_\kappa(x,z) \leq \frac{1}{l}$. \\
Finally, if $d_\kappa(x,y) = \frac{1}{l}$ and $d_\kappa(y,z) = \frac{1}{k}$ for some $l,k \in \mathbb{N}$, we assume without loss of generality that $l \geq k$, so $\frac{1}{l} \leq \frac{1}{k}$. Therefore, we have that $(x,y),(y,z) \in D_k$, so $(x,z) \in D_k$ and hence $d_\kappa(x,z) \leq \frac{1}{k} = d_\kappa(y,z)$. \\
Now let 
\begin{align*}
M := \{ sup(d_{\kappa_1},...,d_{\kappa_n}) \mid n \in \mathbb{N}, \kappa_1,...,\kappa_n \in AK(\mathcal{B}) \}.
\end{align*}
By Proposition 2.2.2(1), $M$ is a system of non-Archimedean pseudo-metrics. We prove that $\mathfrak{D} = \mathfrak{D}_M$, which splits into the following two assertions:

\begin{itemize}
\item[(i)] Given $D \in \mathcal{B}$, there exist $d \in M$ and $\varepsilon > 0$ such that $D^d_\varepsilon \subseteq D$.
\item[(ii)] Given $ d \in M$ and $\varepsilon > 0$, there exists $D \in \mathfrak{D}$ such that $D \subseteq D^d_\varepsilon$.
\end{itemize}
For the proof of (i), let $D \in \mathcal{B}$. Then $\kappa := (D_n)_{n \in \mathbb{N}}$ with $D_1 = X \times X$ and $D_n = D$, for all $n \geq 2$, is an element of $AK(B)$. Now let $\varepsilon > 0$ such that $\varepsilon < \frac{1}{2}$ and define $d = d_\kappa$. For $(x,y) \in D_\varepsilon^d$ we have $d(x,y) < \varepsilon < \frac{1}{2}$, so $(x,y) \in D_2 = D$. \\
Now consider $d = sup(d_1,...,d_n)$ where $d_i = d_{\kappa_i}$ and $\kappa_i = (D_m^{(i)})_{m \in \mathbb{N}} \in AK(\mathcal{B})$ for all $i \in \{1,...,n\}$. Moreover, let $\varepsilon > 0$ and $k \in \mathbb{N}$ such that $\frac{1}{k} < \varepsilon$. The desired relation is
\begin{align*}
D := \bigcap \limits_{i=1}^n D_k^{(i)} \in \mathfrak{D}.
\end{align*}
To see this, let $(x,y) \in D$. Then we have $d_i(x,y) \leq \frac{1}{k} < \varepsilon$ for all $i \in \{1,...,n\}$, so $d(x,y) < \varepsilon$, i.e. $(x,y) \in D^\varepsilon_d$. \\
So $M$ is a system of non-Archimedean pseudo-metrics such that $\mathfrak{D} = \mathfrak{D}_M$.
\end{proof}

\begin{theorem}
For a diagonal uniformity $\mathfrak{D}$ on $X$ the following are equivalent:
\begin{itemize}
\item[(A)] $\mathfrak{D}$ is non-Archimedean.
\item[(B)] There exists a system $M$ of non-Archimedean pseudo-metrics on $X$ that induces the uniformity~$\mathfrak{D}$.
\item[(C)] The corresponding covering uniformity $\mu_\mathfrak{D}$ of $\mathfrak{D}$ has a basis consisting of partitions of $X$.
\end{itemize}
\end{theorem}

\begin{proof}
The equivalence of (A) and (C) follows immediately from Lemma 2.1.2. \\
That (B) implies (A) follows from Proposition 2.2.2. \\
The implication from (A) to (B) was just proven in Proposition 2.2.3.
\end{proof}

\section{An equivalent seperation axiom}

\theoremstyle{definition}
\newtheorem{2.1}{Definition}[section]
\begin{2.1} 
Let $(X, \tau)$ be a topological space.
\begin{itemize}
\item[(1)]$X$ is called \textit{uniformizable by a non-Archimedean uniformity} if there exists a non-Archimedean uniformity on $X$ that induces $\tau$.
\item[(2)] $X$ is said to satisfy $\text{T}_\text{A}$ if for every closed subset $A \subseteq X$ and every $x \in X \setminus A$ there are open sets $U_1,U_2 \subseteq X$ such that 
\begin{itemize}
\item[(i)] $U_1 \cap U_2 = \emptyset $, $U_1 \cup U_2 = X$,
\item[(ii)] $A \subseteq U_1$, $x \in U_2$.
\end{itemize}
\item[(3)] $X$ is said to be \textit{zero-dimensional} (\textit{with respect to the small inductive dimension}), if it has a basis of clopen sets. 
\end{itemize}
\end{2.1}

The equivalence of (1) and (3) in the following theorem has already been shown by Banaschewski~\cite{Banaschewski} for Hausdorff spaces.

\theoremstyle{plain}

\newtheorem{2.2}[2.1]{Theorem}
\begin{2.2}
For a topological space $(X,\tau)$, the following are equivalent:
\begin{itemize}
\item[(1)] $X$ is uniformizable by a non-Archimedean uniformity.
\item[(2)] $X$ satisfies $\text{T}_\text{A}$.
\item[(3)] $X$ is zero-dimensional.
\end{itemize}
\end{2.2}

\begin{proof}
"(2) $\Rightarrow$ (1)" Now assume that $X$ satisfies $\text{T}_\text{A}$. Let $S = \{f: X \to \{0,1\} \mid f \text{ is continuous} \}$, where $\{0,1\}$ carries the discrete topology. For every $f \in S$, we define an equivalence relation 
\begin{align*}
D_f = \{ (x,y) \in X \times X \mid f(x) = f(y) \}.
\end{align*}
It is immediate that $\mathcal{B} := \{D_{f_1} \cap ... \cap D_{f_n} \mid f_i \in S \}$ is a basis consisting of equivalence relations for a non-Archimedean uniformity $\mathfrak{D}$ on $X$. We denote by $\sigma$ the topology on $X$ induced by $\sigma$ and assert that $\tau = \sigma$. \\
Let $A \subseteq X$ be closed with respect to $\tau$ and $x \in X \setminus A$. Let $U_1, U_2 \subseteq X$ open with respect to $\tau$ such that $U_1 \cap U_2 = \emptyset$, $U_1 \cup U_2 = X$, $x \in U_2$, $A \subseteq U_1$ and let $f_x: X \to \{0,1\}$ be such that $f_x$ takes the value $0$ on $U_2$ and the value $1$ on $U_1$. Then clearly $f_x$ is continuous with respect to $\tau$. Per definition,
\begin{align*}
D_{f_x}[x] = \{y \in X \mid f_x(y) = f_x(x) = 0 \} = U_2
\end{align*}
is open with respect to $\sigma$ and disjoint from $A$. Since $x \in X \setminus A$ was arbitrarily chosen, we conclude that
\begin{align*}
A = X \setminus \bigcup \limits_{x \in X \setminus A} D_{f_x}[x]
\end{align*}
is closed with respect to $\sigma$. So $\sigma$ is finer than $\tau$.\\
For the reverse direction, it suffices to show that $D[x]$ is open with respect to $\tau$ for all $x \in X$ and $D \in \mathcal{B}$. But since $D = D_{f_1} \cap ... \cap D_{f_n}$ for some $f_1,...,f_n \in S$ and $D[x] = D_{f_1}[x] \cap ... \cap D_{f_n}[x]$ for all $x \in X$, we just have to show this for $D_f$ with $f \in S$. There, for every $x \in X$, we have $D_f[x] = \{y \in X \mid f(x) = f(y) \} = f^{-1}(\{f(x)\})$. Since $f$ is continuous and $\{f(x)\}$ is open, it follows that $D_f[x] \in \tau$. \\
So $X$ is uniformizable by a non-Archimedean uniformity. \\
"(3) $\Rightarrow$ (2)" Let $A \subseteq X$ be closed and $x \in X \setminus A$. Since $X \setminus A$ is open and therefore a union of clopen sets, there exists some $U_2 \subseteq X \setminus A$ that is clopen and contains $x$. Set $U_1 = X \setminus U_2$.
\end{proof}

\noindent
Note that, although $\text{T}_\text{A}$ implies $\text{T}_{3\frac{1}{2}}$, it does not fit into the classical hierachy of seperation axioms, as normal spaces are not necessarily $\text{T}_\text{A}$. Consider for instance $\mathbb{R}$ carrying the canonical topology. This space is normal, but it is connected, hence does not satisfy $\text{T}_\text{A}$.

\section{Pseudo-metrizability}

\noindent
We now proceed with a characterization of the case where a uniform space is pseudo-metrizable by a single non-Archimedean pseudo-metric. 

\newtheorem{3.1}{Theorem}[section]
\begin{3.1}
Let $(X,\mathfrak{D})$ be a uniform space. Then the following are equivalent:
\begin{itemize}
\item[(A)] $(X,\mathfrak{D})$ is pseudo-metrizable by a non-Archimedean pseudo-metric.
\item[(B)] $\mathfrak{D}$ is induced by a countable system of non-Archimedean pseudo-metrics.
\item[(C)] $\mathfrak{D}$ possesses a countable basis consisting of equivalence relations.
\end{itemize}
\end{3.1}

\begin{proof}
(A)$\Rightarrow$(B) is trivial.\\
For (B)$\Rightarrow$(C), let $(d_n)_{n \in \mathbb{N}}$ be a countable system of non-Archimedean pseudo-metrics inducing $\mathfrak{D}$. Then $(D^{d_n}_{1/m})_{n,m \in \mathbb{N}}$ is a countable basis of $\mathfrak{D}$ consisting of equivalence relations. \\
Now let $(E_n)_{n \in \mathbb{N}}$ be a basis for $\mathfrak{D}$ consisting of equivalence relations and assume without restriction that $E_1 = X \times X$. We get another countable basis $(D_n)_{n \in \mathbb{N}}$ (of equivalence relations) by applying the properties of a basis of a uniformity and iterating the following recursion:
\begin{center}
$D_1 := E_1$ \\
$\text{Choose } D_{i+1} \in \{E_n \mid n \in \mathbb{N} \} \text{ such that } D_{i+1} \subseteq D_i \cap E_{i+1},$
\\ $\text{for } i \in \mathbb{N}.$
\end{center}
Note that $(D_n)_{n \in \mathbb{N}}$ is a decreasing chain, i.e. $X \times X = D_1 \supseteq D_2 \supseteq D_3 \supseteq ...$, what allows us to define a non-Archimedean pseudo-metric as follows:
\begin{align*}
d : X \times X &\to \mathbb{R}_{\geq 0} \\
			(x,y)     &\mapsto \begin{cases}
								0 \text{ if } \forall n \in \mathbb{N}: (x,y) 																			\in D_n\\
								\frac{1}{n} \text{ if } n = max\{ m \in 														\mathbb{N} \mid (x,y) \in D_m \}.
								\end{cases} 
\end{align*}
To see that $d$ induces $\mathfrak{D}$, let $\varepsilon > 0$ and $n \in \mathbb{N}$ such that $\frac{1}{n} < \varepsilon$. For $(x,y) \in D_n$, we have $d(x,y) \leq \frac{1}{n} < \varepsilon$, which implies $(x,y) \in D_\varepsilon^d$. \\
On the other hand, let $n \in \mathbb{N}$. For $(x,y) \in D^d_\frac{1}{n+1}$, it follows that $d(x,y) \leq \frac{1}{n+1} < \frac{1}{n}$, hence $(x,y) \in D_n$.
\end{proof}

\end{document}